\documentclass[11pt,reqno]{amsart}
\usepackage{amsmath}
\usepackage{amssymb}
\usepackage{amsthm}
\usepackage{mathrsfs}
\usepackage{enumerate}
\usepackage[english]{babel}
\usepackage{hyperref}
\usepackage[all,cmtip]{xy}
\entrymodifiers={+!!<0pt,\fontdimen22\textfont2>}

\newtheorem{thm}{Theorem}[section]
\newtheorem{lem}[thm]{Lemma}
\newtheorem*{thm-i}{Theorem}
\newtheorem{cor}[thm]{Corollary}
\newtheorem{prop}[thm]{Proposition}

\theoremstyle{definition}

\newtheorem{rem}[thm]{Remark}
\newtheorem{defn}[thm]{Definition}

\newtheorem{ex}[thm]{Example}

\newtheorem*{acknowledgments*}{Acknowledgments}

\numberwithin{equation}{section}

\theoremstyle{remark}

\mathchardef\ordinarycolon\mathcode`\: 
\def\vcentcolon{\mathrel{\mathop\ordinarycolon}} 
\providecommand*\coloneqq{\mathrel{\vcentcolon\mkern-1.2mu}=}

\def\Z{{\mathbb Z}} %integer
\def\R{{\mathbb R}} %real
\def\C{{\mathbb C}} %complex
\DeclareMathOperator\x{\otimes}
\DeclareMathOperator\omax{\x_{\max}}
\DeclareMathOperator\omin{\x_{\min}}

\DeclareMathOperator\Tor{Tor}
\def\KK{{K\!K}}

\def\cast{$C^{*}$}

\def\id{{\mathrm{id}}}
\def\Compact{\mathcal K}

\def\Kunneth{K\"{u}nneth }
\def\N{\mathcal N}
\def\Nmax{\N_{\max}}
\def\Nmin{\N_{\min}}

\begin{document}

\title{A note on the K\"{u}nneth theorem for nonnuclear \cast-algebras}
\author{Otgonbayar Uuye}
\date{\today}
\address{
School of Mathematics\\ 
Cardiff University\\
Senghennydd Road\\
Cardiff, Wales, UK.\\
CF24 4AG}
\email{UuyeO@cardiff.ac.uk}
%\subjclass[2000]{Primary (58J42); Secondary (58B34)}

\begin{abstract} In this mostly expository note, we revisit the \Kunneth theorem in $K$-theory of nonnuclear \cast-algebras. We show that, using examples considered by Skandalis, there are algebras satisfying the \Kunneth theorem for the minimal tensor product but not for the maximal tensor product and vice versa.
\end{abstract}

\maketitle

\section{Introduction}

Let $A$ be a \cast-algebra. Suppose that $A$ is {\em nuclear}, that is, for any \cast-algebra $B$, the algebraic tensor product $A \odot B$ admits a unique \cast-norm. Let $A \x B$ denote the completion. Let $K_{*}$ denote the $\Z/2\Z$-graded topological $K$-theory.

The {\em \Kunneth theorem}, first studied by Atiyah in the abelian case \cite{MR0150780} and Schochet in the general (nuclear) case \cite{MR650021}, concerns the question of to what extent the natural $\Z/2\Z$-graded product map 
	\begin{equation}\label{eq alpha}
	\xymatrix@1{
	\alpha\colon K_{*}(A) \otimes K_{*}(B) \ar[r] & K_{*}(A \x B)}
 	\end{equation}
is an {\em isomorphism}. The following is the original statement of Schochet. See also \cite{MR1656031,MR2100669}.
\begin{thm}[{\cite{MR650021}}]\label{thm Kunneth}
Let $A$ and $B$ be \cast-algebras with $A$ in the smallest subcategory of the category of separable nuclear C*-algebras which contains the separable Type I algebras and is closed under the operations of taking ideals, quotients, extensions, inductive limits, stable isomorphism, and crossed products by $\Z$ and by $\R$. Then there is a natural $\Z/2\Z$-graded \Kunneth exact sequence
	\begin{equation}\label{eq Kunneth exact sequence}
	0 \to K_{*}(A) \otimes K_{*}(B) \overset{\alpha}{\to} K_{*}(A \x B) \to \Tor(K_{*}(A), K_{*}(B)) \to 0.
	\end{equation}
\end{thm}

\begin{rem}
\begin{enumerate}
\item It was shown in \cite{MR894590} that the \Kunneth exact sequence (\ref{eq Kunneth exact sequence}) always splits.
\item It is an open problem whether {\em all} separable {\em nuclear} \cast-algebras satisfy the \Kunneth exact sequence (\ref{eq Kunneth exact sequence}).
\end{enumerate}
\end{rem}  

For general \cast-algebras $A$ and $B$, the algebraic tensor product $A \odot B$ can be completed to a \cast-algebra in various ways. In this note, we consider the {\em maximal} tensor product $A \omax B$ and the {\em minimal} tensor product $A \omin B$ (see \cite{MR1873025, MR2391387}). We let $\pi = \pi_{A, B}$ denote the natural map	
	\begin{equation}\label{max to min}
	\xymatrix@1{\pi\colon A \omax B \ar@{->>}[r] & A \omin B}.
 	\end{equation}
%The \Kunneth theorem concerns the question of to what extent the natural product maps $\alpha_{\max}$ and $\alpha_{\min}$
%We say that $A$ is {\em nuclear} if the map (\ref{max to min})	is an isomorphism for all $B$. 

In \cite{MR953916}, Skandalis constructed examples of algebras $A$ and $B$ such that the map $\pi_{A, B}$  is {\em not} isomorphic on $K$-theory (see Example~\ref{ex min-max}). Hence, for the \Kunneth theorem for general \cast-algebras, we need to distinguish the tensor products $\omax$ and $\omin$.  

%Then we have two product maps $\alpha_{\max}$ and $\alpha_{\min}$:
%	\begin{equation}\label{eq alpha max min}
%	\xymatrix{
%	& K_{*}(A \omax B) \ar[dd]^{\pi_{*}}\\
%	K_{*}(A) \otimes K_{*}(B) \ar[ru]^{\alpha_{\max}} \ar[rd]_{\alpha_{\min}}&  \\
%	& K_{*}(A \omin B)\\}.
% 	\end{equation}
%If $A$ is nuclear, then $\pi$ is an isomorphism and we can write $\x \coloneqq \omax = \omin$ and $\alpha \coloneqq \alpha_{\max} = \alpha_{\min}$.
 
We consider the \Kunneth theorem for $\omin$ in Section~\ref{sec min} and $\omax$ in Section~\ref{sec max}. Counterexamples are discussed in Section~\ref{sec cex}. We note that these counterexamples are not new and were considered in \cite{MR953916,MR1143449,MR2100669,MR2058474}. 

For the convenience of the reader, we start by recalling the mapping cone construction and the Puppe exact sequence in Section~\ref{sec Puppe}. We remark that we do not assume that our \cast-algebras are separable, since it is an unnatural and unnecessary restriction from our point of view. However, we do restrict, for simplicity, to separable algebras when we deal with $\KK$ or $E$-theory.

In Appendix~\ref{app skandalis}, we sketch Skandalis' examples.

%Skandalis studied the failure of the \Kunneth theorem in \cite{MR953916,MR1143449}. We reproduce his argument in Appendix for the convenience of the reader.

%Appendix~\ref{app CFO} concerns the construction of a bivariant $K$-theory. We omit the proofs.

\begin{acknowledgments*} The author is supported by an EPSRC fellowship. The author wishes to thank Takeshi Katsura for interesting discussions on the topic.
\end{acknowledgments*}

\section{Mapping Cones}\label{sec Puppe}

We recall the Puppe exact sequence in $K$-theory. All the material in this section are well-known. See \cite{MR658514,MR757510,MR1656031,MR2340673}.

Let
	\begin{equation}
	C_{0}[0, 1) \coloneqq \{f \colon [0, 1] \to \C \mid f(1) = 0\}
	\end{equation}
and let 
	\begin{equation}
	\mathrm{ev}_{0}\colon C_{0}[0, 1) \to \C, \quad f \mapsto f(0)
	\end{equation}
denote the evaluation map at $0 \in [0, 1)$.  
\begin{defn} Let $\phi\colon A \to B$ be a $*$-homomorphism. The {\em mapping cone} $C_{\phi}$ of $\phi$ is the pullback
	\begin{equation}
	\xymatrix{C_{\phi} \ar[r] \ar[d] & C_{0}[0, 1) \x B \ar[d]^{\mathrm{ev}_{0} \x \id_{B}}\\
	A \ar[r]^{\phi} & B}.
	\end{equation}
\end{defn}

%\begin{prop}[{Pedersen \cite[Proposition 3.1]{MR1716199}}] Consider a commutative diagram of \cast-algebras:
%	\begin{equation}\label{eq com diag}
%	\xymatrix{
%	X \ar[r]^{\gamma} \ar[d]^{\delta} & B \ar[d]^{\beta}\\
%	A \ar[r]^{\alpha} & C
%	}
%	\end{equation}
%Suppose that $\alpha$ is surjective. Then (\ref{eq com diag}) is a pullback diagram if and only if 
%\begin{enumerate}
%\item $\gamma$ is surjective and 	 
%\item $\delta$ restricts to an isomorphism $\delta\colon\ker\delta \cong \ker \alpha$.
%\end{enumerate}
%\qed
%\end{prop}

\begin{thm}[Puppe Exact Sequence]\label{thm puppe}
Let $\phi\colon A \to B$ be a $*$-homomorphism. Then there is a {\em natural} $6$-term exact sequence
	\begin{equation}
	\xymatrix{
	K_{0}(C_{\phi}) \ar[r] & K_{0}(A) \ar[r]^-{\phi_{*}} & K_{0}(B) \ar[d]\\
	K_{1}(B) \ar[u]& \ar[l]_-{\phi_{*}}  K_{1}(A) & \ar[l] K_{1}(C_{\phi})\\
	}
	\end{equation}
\end{thm}
\begin{proof} See \cite[Theorem 3.8 \& 4.1]{MR658514} or \cite[Theorem 2.38]{MR2340673}. %Here is an alternative proof:
%
%By Proposition~\ref{prop cone tensor algebra}, the mapping cone of 
%	\begin{equation}
%	\phi \x \id_{\mathcal{K}}\colon A \x \mathcal{K} \to B \x \mathcal{K}
%	\end{equation}
%is the algebra $C_{\phi} \x \mathcal{K}$, where $\mathcal{K}$ is the \cast-algebra of compact operators on a separable Hilbert space of infinite dimension. Now the claim follows from \cite[Theorem 3.8 \& 4.1]{MR658514} and Bott periodicity. 
\end{proof}

\begin{cor}\label{cor puppe} Let $\phi\colon A \to B$ be a $*$-homomorphism. Then $\phi$ induces an isomorphism $\phi_{*}\colon K_{*}(A) \cong K_{*}(B)$ if and only if $K_{*}(C_{\phi}) = 0$. \qed
\end{cor}

The following properties of the mapping cone are folklores and follow immediately from Proposition~\ref{prop pullback}.

\begin{prop}\label{prop cone tensor algebra} Let $\phi\colon A \to B$ be a $*$-homomorphism and let $D$ be a \cast-algebra. Then we have natural isomorphisms
	\begin{align}
	C_{\phi} \omax D &\cong C_{\phi \omax \id_{D}},\\
	C_{\phi} \omin D &\cong C_{\phi \omin \id_{D}}.
	\end{align}
\qed
\end{prop}
%\begin{proof} For the minimal tensor product $\omin$, note that the evaluation map $\mathrm{ev}_{0} \colon C_{0}[0, 1) \to \C$ admits a completely positive section.
%\end{proof}

\begin{prop}\label{prop cone crossed prod} Let $G$ be a locally compact topological group and let $\phi\colon A \to B$ be a morphisms of $G$-\cast-algebras. Then $C_{\phi}$ is a $G$-\cast-algebra and
	\begin{align}
	C_{\phi} \rtimes G \cong C_{\phi \rtimes G}.
	\end{align}
	\qed
\end{prop}

\begin{prop}\label{prop pullback} Consider a commutative diagram
	\begin{equation}
	\xymatrix{
	0 \ar[r] &I \ar[r] \ar@{=}[d] & X \ar[r] \ar[d] & A \ar[d] \ar[r] & 0\\
	0 \ar[r] &I \ar[r] &D \ar[r] & B  \ar[r]& 0
	}
	\end{equation}
Suppose that the lower row is exact. Then the right-hand square is a pullback diagram if and only if the upper row is exact.
\end{prop}
\begin{proof} See \cite[Proposition 3.1]{MR1716199}.
\end{proof}

\section{The Minimal Tensor Product}\label{sec min}
%%%%%%%%%%%%%%%%%%%%%%%%%%%%

The following is the K\"{u}nneth theorem for the minimal tensor product. The equivalence (\ref{Kfree}) $\Leftrightarrow$ (\ref{Kfull}) is shown in \cite{MR2100669}. The condition (\ref{Kzero}) is an analogue of the condition (iii) of \cite[Proposition 5.3]{MR953916}.
\begin{thm}[K\"{u}nneth theorem for $\omin$]\label{thm min-Kunneth} Let $A$ be a \cast-algebra. Then the following conditions on $A$ are equivalent.
\begin{enumerate}
%\item\label{Ksep} For any separable \cast-algebra $B$, if $K_{*}(B) = 0$ then $K_{*}(A \x B) = 0$.
\item\label{Kzero} For any \cast-algebra $B$, if $K_{*}(B) = 0$ then $K_{*}(A \omin B) = 0$.
\item\label{Kfree} For any  \cast-algebra $B$, if $K_{*}(B)$ is free then the product map 	\[\alpha_{\min}\colon K_{*}(A) \otimes K_{*}(B) \to K_{*}(A \omin B)\]
is an isomorphism.
\item\label{Kfull} For any \cast-algebra $B$, there is a (natural) short exact sequence
	%\begin{equation} 
 	\[
	0 \to K_{*}(A) \otimes K_{*}(B) \overset{\alpha_{\min}}{\to} K_{*}(A \omin B) \to \Tor(K_{*}(A), K_{*}(B)) \to 0.
	\]
	%\end{equation}	
\end{enumerate}
\end{thm}
\begin{proof} The implications (\ref{Kfull}) $\Rightarrow$ (\ref{Kfree}) $\Rightarrow$ (\ref{Kzero}) are clear. The implication (\ref{Kfree}) $\Rightarrow$ (\ref{Kfull}) is due to Schochet and follows from the existence of a geometric resolution (cf. Proof of \cite[Theorem 4.1]{MR650021} or \cite[Theorem 3.3]{MR2100669}). 

For the implication (\ref{Kzero}) $\Rightarrow$ (\ref{Kfree}), let $B$ be a \cast-algebra with $K_{*}(B)$ free. Let $\Sigma^{n} \coloneqq C_{0}(\R^{n})$, $n \ge 0$. In the following, we abbriviate $\omin$ by $\x$. 

 Since $K_{*}(B)$ is free, there is an abelian \cast-algebra of the form $D = \oplus_{\Lambda_{1}} \Sigma^{2} \bigoplus \oplus_{\Lambda_{2}} \Sigma^{3}$ and a $*$-homomorphism $\varphi\colon D \to \Sigma^{2} \otimes B \otimes \Compact$ inducing isomorphism in $K$-theory, where $\Compact$ is the \cast-algebra of compact operators on a suitable Hilbert space. Let $C_{\varphi}$ denote the mapping cone of $\varphi$. Then $K_{*}(C_{\varphi}) = 0$ by Corollary~\ref{cor puppe}. Since $A \x C_{\varphi} \cong C_{\id_{A} \x \varphi}$, we see that $K_{*}(C_{\id_{A} \x \varphi}) = K_{*}(A \x C_{\varphi}) = 0$ by (\ref{Kzero}),  hence $\id_{A} \x \varphi$ induces an isomorphism in $K$-theory, again by Corollary~\ref{cor puppe}. The top map in the following commutative diagram is clearly an isomorphism, thus it follows that $\alpha_{\min}$ is an isomorphism for $(A, \Sigma^{2} \x B \x \Compact)$.
	\begin{equation}
	\xymatrix{
	K_{*}(A) \otimes K_{*}(D) \ar[r]^{\cong} \ar[d]^{\id_{K_{*}(A)} \x \varphi_{*}} & K_{*}(A \x D) \ar[d]^{(\id_{A} \x \varphi)_{*}}\\
	K_{*}(A) \otimes K_{*}(\Sigma^{2} \x B \x \Compact) \ar[r]^-{\alpha_{\min}} & K_{*}(A \x \Sigma^{2} \x B \x \Compact)\\
%	K_{*}(A) \otimes K_{*}(q\C \x B) \ar[r]^{\alpha} & K_{*}(A \x q\C \x B)\\
%	K_{*}(A) \otimes K_{*}(B) \ar[r]^{\alpha} & K_{*}(A \x B)\\
	}
	\end{equation}
Now Bott periodicity	completes the proof.
\end{proof}

\begin{rem}[Separable algebras]\label{rem sep min-Kunneth} Let $A$ be a \cast-algebra. The proof of Theorem~\ref{thm min-Kunneth} shows that the following conditions are equivalent\footnote{If $A$ is also separable, the proof can be shortened using $\KK$-theory.}. See also \cite[Theorem 3.3]{MR2100669}.
\begin{enumerate}
\item[(1')]\label{sepKzero} For any {\em separable} \cast-algebra $B$, if $K_{*}(B) = 0$ then $K_{*}(A \omin B) = 0$.
\item[(2')]\label{sepKfree} For any {\em separable} \cast-algebra $B$, if $K_{*}(B)$ is free then the product map
	\[\alpha_{\min}\colon K_{*}(A) \otimes K_{*}(B) \to K_{*}(A \omin B)\]
is an isomorphism. 
\item[(3')]\label{sepKfull} For any {\em separable} \cast-algebra $B$, there is a (natural) short exact sequence
	%\begin{equation} 
 	\[
	0 \to K_{*}(A) \otimes K_{*}(B) \overset{\alpha_{\min}}{\to} K_{*}(A \omin B) \to \Tor(K_{*}(A), K_{*}(B)) \to 0.
	\]
	%\end{equation}	
\end{enumerate}
Moreover, it is easy to see that the a priori weaker condition (3') is equivalent to (\ref{Kfull}). Indeed, write $B$ as the inductive limit of its separable \cast-subalgebras under inclusions: $B \cong \lim_{B' \subseteq B} B'$, $B'$ separable. Then 
	\begin{equation}
	A \omin B \cong \lim_{B' \subseteq B} A \omin B'
	\end{equation}
and the implication (3') $\Rightarrow$ (\ref{Kfull}) follows from the continuity of $K$-theory and the fact that tensor products of abelian groups commute with direct limits. It follows that all six conditions are equivalent, hence we may restrict to $B$ separable in Theorem~\ref{thm min-Kunneth}.
\end{rem}

\begin{defn} 
Let $\Nmin$ denote the class of \cast-algebras $A$ satisfying the equivalent conditions of Theorem \ref{thm min-Kunneth} and Remark~\ref{rem sep min-Kunneth}. %Similarly for $\Nmin$.
\end{defn}

Now we survey some results about $\Nmin$ and list some examples.

\begin{lem}[{\cite[Section 2]{MR650021}, \cite[23.4]{MR1656031} or \cite[Lemma 4.4]{MR2100669}}]\label{lem properties of Nmin} The class $\Nmin$ enjoys the following properties. 
\begin{enumerate}
\item If $A \in \Nmin$ and $B$ is Morita dominated by $A$, then $B \in \Nmin$. In particular, $\Nmin$ is stable under Morita equivalence.
\item\label{item min 23} In a {\em semi-split} short exact sequence of \cast-algebras, if two of the algebras are in $\Nmin$ , then so is the third.
\item If $A$, $B \in \Nmin$, then $A \omin B \in \Nmin$.
\item If $A = \lim A_{i}$, such that all structure maps are injective and all $A_{i} \in \Nmin$, then $A \in \Nmin$.
\end{enumerate}
\qed
\end{lem}

The following result of Chabert-Echterhoff-Oyono-Oyono generalises the $\Z$ and $\R$ case considered by Schochet. 
\begin{thm}[{\cite[Corollary 0.2]{MR2100669}}] Let $G$ be a second countable locally compact topological group satisfying the Baum-Connes conjecture with coefficients (cf.\ \cite[Conjecture 9.6]{MR1292018}). Let $A$ be a separable $G$-algebra. If $A \rtimes K \in \Nmin$ for all compact subgroups $K \subseteq G$, then $A \rtimes_{\mathrm{red}} G \in \Nmin$. 
\end{thm}

The following is essentially a repackaging of their proof.
\begin{proof} By \cite[Theorem 9.3]{MR2193334}, the Baum-Connes conjecture with coefficients can be restated as follows:

For any $G$-\cast-algebra $A$, if $K_{*}(A \rtimes K) = 0$ for all compact subgroups $K \subseteq G$, then $K_{*}(A \rtimes_{\mathrm{red}} G) = 0$.

Now the proof is easily completed by appealing to Theorem~\ref{thm min-Kunneth}.
\end{proof}

%\begin{defn} Let $A$ be a separable \cast-algebra. We say that $A$ {\em satisfies the \Kunneth theorem for the maximal (respectively minimal) tensor product} if the map $\alpha_{\max}$ (respectively $\alpha_{\min}$) is an isomorphism for any separable \cast-algebra $B$ with $K_{*}(B)$ free. We denote by $\Nmax$ (respectively $\Nmin$) the class of \cast-algebras satisfying the \Kunneth theorem for the maximal (respectively minimal) tensor product.
%\end{defn}

%The following is a version of the \Kunneth theorem, due to Schochet.

%\begin{thm}[Schochet] Let $A$ be a separable \cast-algebra. 
%
%Suppose that $A$ is $\KK$-equivalent to an abelian separable \cast-algebra. Then for any separable \cast-algebra $B$, the natural map $\pi$ induces an isomorphism
%	\begin{equation}
%	\pi_{*}\colon K_{*}(A \omax B) \overset{\cong}{\longrightarrow} K_{*}(A \omin B)
% 	\end{equation}
%and there is a natural short exact sequence
%	\begin{equation} 
% 	0 \to K_{*}(A) \otimes K_{*}(B) \to K_{*}(A \omax B) \to \Tor(K_{*}(A), K_{*}(B)) \to 0
%	\end{equation}	
%\end{thm}

%Hence we may say that $\Nmin$ is the class of \cast-algebras satisfying the {\em \Kunneth theorem in $K$-theory for the minimal tensor product}.

%\begin{thm}[Schochet] The class of \cast-algebras that satisfy the \Kunneth theorem has the following properties:
%\begin{enumerate}
%\item It contains all separable type I algebras.
%\item It is closed under countable inductive limits.
%\item It satisfies the 2-out-of-3 property for any extension.
%\item It is closed under $\KK$-equivalence.
%\end{enumerate} 
%\end{thm}

\begin{ex} 
\begin{enumerate}
\item Type I algebras are in $\Nmin$ (Schochet \cite[Theorem 2.13]{MR650021}).
\item Any separable \cast-algebra in the {\em bootstrap category} of \cast-algebras $\KK$-equivalent to an abelian \cast-algebra is in $\Nmin$. (Rosenberg-Shochet \cite{MR894590}). The groupoid \cast-algebra of an amenable groupoid (Tu \cite[Proposition 10.7]{MR1703305}) and the full and reduced group \cast-algebras $C^{*}(G)$ and $C^{*}_{\lambda}(G)$ of an almost-connected group (Chabert-Echterhoff-Oyono-Oyono \cite[Proposition 5.1]{MR2100669}) are in the bootstrap category, hence in $\Nmin$.
\item Let $G$ be a separable locally compact group such that the component group $G/G_{0}$ satisfies the Baum-Connes conjecture with coefficients. Then the reduced group algebra $C^{*}_{\lambda}(G)$ is in $\Nmin$. (Chabert-Echterhoff-Oyono-Oyono \cite[Corollary 0.3]{MR2100669}).
\end{enumerate}
\end{ex}

 However, as we see below, there are non-nuclear (in fact non-exact, see Remark~\ref{rem KK-exactness}) \cast-algebras that are not in $\Nmin$.

\begin{defn} We say that a \cast-algebra $A$ is {\em $K$-exact} if the functor $B \mapsto K_{0}(A \omin B)$ is half-exact.
\end{defn}

Clearly, exact \cast-algebras are $K$-exact. %, hence the following give examples of algebras not in $\Nmin$.

The following remark is due to Skandalis (c.f.\ \cite[Remark 4.3]{MR2100669}).
\begin{rem}\label{rem sep-K-exactness} Associated to an extension
	\begin{equation}
	0 \to I \to B \to D \to 0,
	\end{equation}
there is a {\em double-cone}\footnote{It is the mapping cone of the inclusion of $I$ into the mapping cone of the quotient map $B \to D$.} algebra $C$ such that the sequence
	\begin{equation}
	K_{*}(A \omin I) \to K_{*}(A \omin B) \to K_{*}(A \omin D)
	\end{equation}
is exact in the middle if and only if $K_{*}(A \omin C) = 0$ (see \cite[p. 335-336]{MR1911663}).

It follows that all \cast-algebras in $\Nmin$ are $K$-exact.

Moreover, the construction of a double-cone is functorial and commutes with inductive limits of extensions. Thus $A$ is $K$-exact if the functor $B \mapsto K_{0}(A \omin B)$ is half-exact on extensions of separable \cast-algebras.
\end{rem}

\begin{ex}\label{ex K-exactness} %Here are some counterexamples to the \Kunneth Theorem \ref{thm min-Kunneth}.
\begin{enumerate} 
\item\label{item CG} Let $\Gamma$ be an infinite countable discrete group with Khazdan property (T), Kirchberg property (F) and Akemann-Ostrand property (AO) (cf. \cite{MR2562137}). Then the full group \cast-algebra $C^{*}(\Gamma)$ is not $K$-exact, hence not in $\Nmin$. (Skandalis \cite{MR1143449}).
\item\label{item M} The product $\prod_{n \ge 1} M_{n}$ is not {\em $K$-exact}, hence not in $\Nmin$. (Ozawa \cite[Theorem A.1]{MR1964549}). 
\end{enumerate}
\end{ex}

\begin{rem}\label{rem KK-exactness} We note that if a separable  \cast-algebra $A$ is not $K$-exact, then it cannot be $\KK$-equivalent to an exact \cast-algebra. 
\end{rem}

\begin{defn} We say that a \cast-algebra $A$ is {\em $K$-continuous} if the functor $B \mapsto K_{0}(A \omin B)$ is continuous i.e.\ commutes with inductive limits.
\end{defn}

Clearly,  \cast-algebras in $\Nmin$ are $K$-continuous. The following is less trivial.
\begin{thm} All $K$-continuous algebras are $K$-exact.
\end{thm}
\begin{proof} Let $A$ be a $K$-continuous \cast-algebra and let $F(B) \coloneqq K_{0}(A \omin B)$. Then by \cite[Theorem 3.11]{MR1262931}, $F$ factors through the asymptotic homotopy category of Connes-Higson \cite{MR1065438}. In particular, for any extension of separable \cast-algebras, the inclusion of the kernel into the mapping cone of the quotient map induces an isomorphism on $F$. It follows that $F$ is half-exact on separable \cast-algebras. The general case follows from Remark~\ref{rem sep-K-exactness}.
\end{proof}

Consequently, Example~\ref{ex K-exactness} give examples of \cast-algebras which are {\em not} $K$-continuous.

\section{The Maximal Tensor Product}\label{sec max}
%%%%%%%%%%%%%%%%%%%%%%%%%%%%

%Let $A$ and $B$ be \cast-algebras. We write $A \omin B = A \x B $ for the {\em minimal} tensor product and $A \omax B$ for the {\em maximal} tensor product. 

%The following theorem is analogues to Theorem~\ref{thm min-Kunneth} and remarks similar to Remark~\ref{rem sep min-Kunneth}  and \ref{rem properties of Nmin} applies.
The maximal tensor product case is analogous, hence we shall be brief.
\begin{thm}[K\"{u}nneth theorem for $\omax$]\label{thm max-Kunneth} Let $A$ be a \cast-algebra. Then the following conditions on $A$ are equivalent.
\begin{enumerate}
%\item\label{Ksep} For any separable \cast-algebra $B$, if $K_{*}(B) = 0$ then $K_{*}(A \x B) = 0$.
\item\label{Kzero-max} For any \cast-algebra $B$, if $K_{*}(B) = 0$ then $K_{*}(A \omax B) = 0$.
\item\label{Kfree-max} For any  \cast-algebra $B$, if $K_{*}(B)$ is free then the product map 	\[\alpha_{\max}\colon K_{*}(A) \otimes K_{*}(B) \to K_{*}(A \omax B)\]
is an isomorphism.
\item\label{Kfull-max} For any \cast-algebra $B$, there is a (natural) short exact sequence
	%\begin{equation} 
 	\[
	0 \to K_{*}(A) \otimes K_{*}(B) \overset{\alpha_{\max}}{\to} K_{*}(A \omax B) \to \Tor(K_{*}(A), K_{*}(B)) \to 0.
	\]
	%\end{equation}
\qed	
\end{enumerate}
\end{thm}

Needless to say, for {\em nuclear} algebras, the \Kunneth theorems \ref{thm min-Kunneth} and \ref{thm max-Kunneth} are equivalent.

\begin{rem}[Separable algebras]\label{rem sep max-Kunneth} Let $A$ be a \cast-algebra. The following conditions are equivalent to the (equivalent) conditions in Theorem~\ref{thm max-Kunneth}.
\begin{enumerate}
\item[(1')]\label{sepKzero-max} For any {\em separable} \cast-algebra $B$, if $K_{*}(B) = 0$ then $K_{*}(A \omax B) = 0$.
\item[(2')]\label{sepKfree-max} For any {\em separable} \cast-algebra $B$, if $K_{*}(B)$ is free then the product map
	\[\alpha_{\max}\colon K_{*}(A) \otimes K_{*}(B) \to K_{*}(A \omax B)\]
is an isomorphism. 
\item[(3')]\label{sepKfull-max} For any {\em separable} \cast-algebra $B$, there is a (natural) short exact sequence
	%\begin{equation} 
 	\[
	0 \to K_{*}(A) \otimes K_{*}(B) \overset{\alpha_{\max}}{\to} K_{*}(A \omax B) \to \Tor(K_{*}(A), K_{*}(B)) \to 0.
	\]
	%\end{equation}	
\end{enumerate}
\qed
\end{rem}

\begin{defn} 
Let $\Nmax$ denote the class of \cast-algebras $A$ satisfying the equivalent conditions of Theorem \ref{thm max-Kunneth} and Remark~\ref{rem sep max-Kunneth}.
\end{defn}

\begin{lem}\label{lem properties of Nmax} The class $\Nmax$ enjoys the following properties. 
\begin{enumerate}
\item If $A \in \Nmax$ and $B$ is Morita dominated by $A$, then $B \in \Nmax$. In particular, $\Nmax$ is stable under Morita equivalence.
\item\label{item max 23} In a short exact sequence of \cast-algebras, if two of the algebras are in $\Nmax$ , then so is the third.
\item If $A$, $B \in \Nmax$, then $A \omax B \in \Nmax$.
\item If $A = \lim A_{i}$ and all $A_{i} \in \Nmax$, then $A \in \Nmax$.
\end{enumerate}
\qed
\end{lem}

We remark that for any \cast-algebra $A$, the functor $B \mapsto K_{0}(A \omax B)$ is half-exact and continuous. Hence we cannot use the same techniques as in Section~\ref{sec min} to construct counterexamples to the \Kunneth theorem for $\omax$. However, see Example~\ref{ex min-max}.

\section{Counterexamples}\label{sec cex}

The counterexamples exploit the difference between Lemma~\ref{lem properties of Nmin}(\ref{item min 23}) and Lemma~\ref{lem properties of Nmax}(\ref{item max 23}). 
\begin{ex}[{$\Nmax \backslash \Nmin \ne \emptyset$; c.f.\ \cite{MR1143449},\cite[Theorem 5.4]{MR2058474}}]\label{ex max-min} Let $C$ be the double-cone of an extension 
	\begin{equation}
	0 \to I \to B \to D \to 0
	\end{equation}
of \cast-algebras (see Remark~\ref{rem sep-K-exactness}). Then for any \cast-algebra $A$, the tensor product $C \omax A$ is the double-cone of the extension 
	\begin{equation}
	0 \to I \omax A \to B \omax A \to D \omax A\to 0.
	\end{equation}
It follows that $K_{*}(C \omax A) = 0$ for all $A$ and $C$ belongs to $\Nmax$.
	
Let $A$ be a non-$K$-exact algebra and let $C$ be the double-cone of an extension for which $K_{*}(A \omin C) \ne 0$. Then $C$ does {\em not} belong to $\Nmin$. Hence $\Nmax \backslash \Nmin \ne \emptyset$.

Here is a concrete example: Let $\Gamma = \mathrm{SL}_{3}(\Z)$ and let 
	\begin{equation}\label{eq Gamma seq}
	0 \to J \to A \to B \to 0
	\end{equation} 
denote the extension of separable commutative $\Gamma$-\cast-algebras of  \cite[Theorem A.1]{MR1964549}. Let $C$ denote the double-cone of (\ref{eq Gamma seq}). Then $C$ is a separable commutative $\Gamma$-\cast-algebra and the full crossed product $C \rtimes \Gamma$ is the double cone of the extension
	\begin{equation}
	0 \to J \rtimes \Gamma \to A \rtimes \Gamma \to B \rtimes \Gamma \to 0
	\end{equation}
(See Proposition \ref{prop cone crossed prod}). Hence $C \rtimes \Gamma \in \Nmax \backslash \Nmin$.  
\end{ex}

The following observation is due to Skandalis \cite{MR953916}.
\begin{lem}\label{lem pi} Let $A$ be a \cast-algebra. Suppose that there is a \cast-algebra $B$ such that the natural map
	\begin{equation}
	\pi\colon A \omax B \to A \omin B
	\end{equation}
does {\em not} induce isomorphism in $K$-theory. Then the following statements are true.
\begin{enumerate}
\item The algebra $A$ fails one of the \Kunneth theorems (\ref{thm min-Kunneth} or \ref{thm max-Kunneth}).
\item If $A$ is separable, then $A$ is not $\KK$-equivalent to a nuclear algebra.
\item If $A$ is separable and exact, then $A$ is not $E$-equivalent to a nuclear algebra.
\end{enumerate}	
\end{lem} 
\begin{proof} 
Enough to note that we may assume that $B$ is separable. See \cite{MR953916, MR2058474}. 
\end{proof}

\begin{ex}[{$\Nmin \backslash \Nmax \ne \emptyset$; c.f.\ \cite[Introduction]{MR2100669}}]\label{ex min-max} Let $\Gamma$ be an infinite countable discrete group with Kazhdan property (T) and Akemann-Ostrand property (AO) (cf.\ \cite{MR2562137}). Then the natural map
	\begin{equation}
	C^{*}_{\lambda}(\Gamma) \omax C^{*}_{\lambda}(\Gamma) \to C^{*}_{\lambda}(\Gamma) \omin C^{*}_{\lambda}(\Gamma) 
	\end{equation}
does not induce isomorphism in $K$-theory (Skandalis \cite{MR953916}). 

We specialise to the case $\Gamma$ a lattice in $\mathrm{Sp}(n, 1)$. Julg proved that $\Gamma$ satisfies the Baum-Connes conjecture with coefficients \cite{MR1903759}. Then by \cite[Corollary 0.2]{MR2100669}, we see that $C^{*}_{\lambda}(\Gamma)$ is in $\Nmin$. Consequently, $C^{*}_{\lambda}(\Gamma)$ is {\em not} in $\Nmax$ by Lemma~\ref{lem pi} and $\Nmin \backslash \Nmax \ne \emptyset$. 
\end{ex}

\begin{ex}\label{ex min+max} Let $A \in \Nmin \backslash \Nmax$ and let $B \in \Nmax \backslash \Nmin$. Then it follows from the 2-out-of-3 property that $A \oplus B$ is neither in $\Nmax$ nor in $\Nmin$.
\end{ex}

Examples~\ref{ex max-min}, \ref{ex min-max} and \ref{ex min+max} answer some of the questions raised by Blackadar in \cite[23.13.2]{MR1656031}.

\appendix
\section{Skandalis' Examples}\label{app skandalis}

We briefly sketch Skandalis' arguments for the convenience of the reader. See \cite{MR953916,MR1143449,MR2058474} for details.

We refer to \cite{MR2562137} for group theoretic terminologies in the following.  
\begin{thm}[Skandalis] Let $\Gamma$ be an infinite countable discrete group with property (T) and property (AO). Then the natural map
	\begin{equation}
	C^{*}_{\lambda}\Gamma \omax C^{*}\Gamma \to C^{*}_{\lambda}\Gamma \omin C^{*}\Gamma 
	\end{equation}
does not induce isomorphism in $K$-theory. If in addition, $\Gamma$ has property (F), then $C^{*}\Gamma$ is not $K$-exact.
\end{thm}
See Example \ref{ex K-exactness}(\ref{item CG}) and Example~\ref{ex min-max}.
%Let $I$, $J$ and $L$ be the kernels of the respective quotient maps in the following diagram. 
\begin{proof}[Sketch of Proof] We assume that $\Gamma$ has (T), (AO) and (F). Then one can construct a commutative diagram of the form
	\begin{equation}
	\xymatrix{
	&&\,\,\,\,\,\,\,I \underset{\min}{\x} C^{*}\Gamma \ar@{_{(}->}[d] \ar@{.>}[dr]^{0}&\\
	\C \ar@{-->}[dr]_-{\text{(T)}} \ar@/_/@{.>}[dd]^(0.79){\ne 0} \ar@/^14pt/@{.>}[rr]^(0.9){\ne 0} \ar@/^15pt/@{.>}[ddrr]^(0.87){0}& \,\,\,\,\,\,\,I \underset{\max}{\x} C^{*}\Gamma \ar@{>->}[d] \ar@{->>}[ur] & L \ar@{>->}[d] \ar@{.>}[r] & \mathcal{K}(l^{2}\Gamma) \ar@{>->}[d]\\
	 & C^{*}\Gamma \underset{\max}{\x} C^{*}\Gamma \ar@{->>}[r] \ar@{->>}[d]& C^{*}\Gamma \underset{\min}{\x} C^{*}\Gamma \ar@{->>}[d] \ar@{-->}[r]^-{\text{(F)}} & \mathcal{B}(l^{2}\Gamma) \ar@{->>}[d]\\
	 J \ar@{>->}[r] \ar@{.>}[d] & C^{*}_{\lambda}\Gamma \underset{\max}{\x} C^{*}\Gamma \ar@{->>}[r] \ar[d]_{\lambda \times \rho} & C^{*}_{\lambda}\Gamma \underset{\min}{\x} C^{*}\Gamma \ar@{-->}[d]_{\text{(AO)}} \ar@{-->}[r]^-{\text{(AO)}}& \mathcal{Q}(l^{2}\Gamma) \ar@{=}[dl]\\
	\mathcal{K}(l^{2}\Gamma) \ar@{>->}[r] & \mathcal{B}(l^{2}\Gamma) \ar@{->>}[r] & \mathcal{Q}(l^{2}\Gamma)
	},
	\end{equation}
where $\xymatrix@1{\bullet\, \ar@{->>}[r] & \,\bullet}$ denotes a quotient map and $\xymatrix@1{\bullet \,\ar@{>->}[r] & \,\bullet\, \ar@{->>}[r] & \,\bullet}$ denotes an extension.
Here $I$ is the kernel of $C^{*}\Gamma \to C^{*}_{\lambda}\Gamma$. %and $J$ and $L$ are the kernels of the respective quotient maps. 

Let $q \in C^{*}\Gamma \omax C^{*}\Gamma$ denote the image of the Kazhdan projection under the diagonal map 
	\begin{equation}
	\Delta\colon C^{*}\Gamma \to C^{*}\Gamma \omax C^{*}\Gamma, \quad \gamma \mapsto \gamma \otimes \gamma.
	\end{equation}

Then the image of $q$ in $C^{*}_{\lambda} \omin C^{*}\Gamma$ is zero, while the image in $B(l^{2}\Gamma)$ is non-zero. Hence $q$ defines non-zero classes in $K_{0}(J)$ and $K_{0}(L)$. Moreover, the composition 
	\begin{equation}
	\xymatrix@1{I \omin C^{*}\Gamma \ar[r] & L \ar[r] & \mathcal{K}(l^{2}\Gamma)}
	\end{equation}
is zero, since the composition $\xymatrix@1{I \omax C^{*}\Gamma \ar[r] & \mathcal{B}(l^{2}\Gamma)}$ is zero.
%in the commutative diagram
%	\begin{equation}
%	\xymatrix{I \omax C^{*}\Gamma \ar@{->>}[r] \ar@{=}[d] & I \omin C^{*}\Gamma \ar[r]^-{\iota} & \mathcal{K}(l^{2}\Gamma) \ar@{>->}[r] & \mathcal{B}(l^{2}\Gamma) \ar@{=}[d]\\
%	I \omax C^{*}\Gamma \ar@{>->}[r] & C^{*}\Gamma \omin C^{*}\Gamma \ar@{->>}[r] &  C^{*}_{\lambda}\Gamma \omax C^{*}\Gamma \ar[r] & \mathcal{B}(l^{2}\Gamma)}
%	\end{equation}
%the second row is zero, and on the first row the first map is surjective and the last map is injective.
	
%	$I \omax C^{*}\Gamma \to I \omin C^{*}\Gamma$ is surjective and $\mathcal{K}(l^{2}\Gamma) \to \mathcal{B}(l^{2}\Gamma)$ is injective.
It follows that the class in $K_{0}(L)$ cannot come from $K_{0}(I \omin C^{*}\Gamma)$. 
\end{proof}

\bibliographystyle{amsalpha}
\bibliography{../BibTeX/biblio}
\end{document}